\documentclass[a4paper,12pt,leqno]{amsart}

\usepackage[active]{srcltx}
\usepackage{verbatim}
\usepackage{epsfig,graphicx,color,mathrsfs}
\usepackage{graphicx}
\usepackage{amsmath,amssymb,amsthm,amsfonts}
\usepackage{amssymb}
\usepackage[english]{babel}

\newtheorem{thm}{Theorem}[section]

\newtheorem{lem}[thm]{Lemma}
\newtheorem{prop}[thm]{Proposition}
\newtheorem{rem}[thm]{Remark}
\newtheorem{definition}[thm]{Definition}

\newcommand{\bremark}{\begin{rem} \textup}
\newcommand{\eremark}{\end{rem} }

\newcommand{\cuad}{{\sqcap\kern-.68em\sqcup}}

\newcommand{\R}{{\mathbb{R}}}

\numberwithin{equation}{section}

\thanks{
AF is partially supported  by ERC-2011-grant: \emph{Epsilon}. \\
BS is partially supported by the
Italian PRIN Research Project 2007: {\em Metodi Variazionali e Topologici
nello Studio di Fenomeni non Lineari}, and is also partially supported by  ERC-2011-grant: \emph{Epsilon}.\\
}
\begin{document}

\parindent 0pc
\parskip 6pt
\overfullrule=0pt

\title[Qualitative properties ]{Qualitative properties and classification \\
of nonnegative solutions to \\  $ -\Delta u=f(u) $ \\ in unbounded  domains \\
when $f(0)<0$}


\author[
A. Farina and B. Sciunzi]{A. Farina and B. Sciunzi}
       \thanks{Address: {\em AF} --
 LAMFA, CNRS UMR 7352, Universit\'{e} de Picardie Jules Verne, 33, Rue Saint-Leu, 80039 Amiens Cedex 1, France.
 E-mail: {\tt
alberto.farina@u-picardie.fr}. {\em BS} --
Universit\`a della Calabria -- (Dipartimento di Matematica e Informatica) -- V. P.
 Bucci 1-- Arcavacata di Rende (CS), Italy. E-mail: {\tt
 sciunzi@mat.unical.it}.\\}


\begin{abstract}
We consider nonnegative solutions to $-\Delta u=f(u)$ in unbounded euclidean domains, where $f$ is merely locally Lipschitz continuous and satisfies $f(0)<0$.
In the half-plane, and without any other assumption on $u$, we prove that $u$ is either one-dimensional and periodic or positive and strictly monotone increasing in the direction orthogonal to the boundary. Analogous results are obtained if the domain is a strip. As a consequence of our main results, we answer affirmatively to a conjecture and to an open question posed by Berestycki, Caffarelli and Nirenberg. We also obtain some symmetry and monotonicity results in the higher-dimensional case.
\end{abstract}

\maketitle

\section{Introduction and main results}

We study qualitative properties of \emph {nonnegative} solutions to $-\Delta u=f(u)$ in unbounded euclidean domains, where $f$ is \emph {merely locally Lipschitz continuous} and satisfies $f(0)<0$.
In particular, we are interested in proving the one-dimensional symmetry, the monotonicity and/or the periodicity of the considered solutions.

Let us start by considering the case of the open half-plane $ \mathbb{R}^2_+: = \{ (x,y)  \in \R^2 : y >0\}$
\begin{equation}\label{E:P}
\begin{cases}
-\Delta u=f(u) & \text{ in } \mathbb{R}^2_+\\
\quad u \geqslant 0 & \text{ in } \mathbb{R}^2_+\\
\quad u=0\,\, &\text{ on } \partial\mathbb{R}^2_+.
\end{cases}
\end{equation}


As a consequence of more general results that we shall state later, we prove

\begin{thm}\label{casteorem}
Assume that $f$ is locally Lipschitz continuous on $[0\,, +\infty)$ with $f(0) <0$.
Let $u\in C^2(\overline{\mathbb{R}^2_+})$ be a nonnegative solution to  \eqref{E:P}. Then, either $u$ is one-dimensional and periodic, or
$u$ is positive and strictly monotone increasing in the direction orthogonal to the boundary with $\partial_y u >0$ in $\mathbb{R}^2_+$.
\end{thm}

The  theorem above provides a complete picture of the situation in our general framework.

\begin{thm}\label{Txxx5}
Let $u\in C^2(\overline{\mathbb{R}^2_+})$ be a nonnegative  solution to
\begin{equation}\label{CongBCN}
\begin{cases}
-\Delta u=u-1 & \text{ in } \mathbb{R}^2_+\\
\quad u \geqslant 0 & \text{ in } \mathbb{R}^2_+\\
\quad u=0\,\, &\text{ on } \partial\mathbb{R}^2_+.
\end{cases}
\end{equation}

Then
\[
u(x,y)\,=\,1-\cos y\,.
\]
\end{thm}

Theorem \ref{Txxx5} provides an affirmative answer to an extended version of a \emph{conjecture} posed by Berestycki, Caffarelli and Nirenberg (see p. 73 of \cite{BCN2}).
It recovers and improves upon a result of \cite{soave} (cf. also \cite{BCN2}) since here no a-priori bound on the solutions is required.

The techniques developed here allows us to consider also problems defined in strips  $\Sigma_{2b}\,:= \{ (x,y)  \in \R^2 : y \in  (0\,,\,2b) \}$, $b >0$.  The following result answers affirmatively to an \emph {open question} raised by
Berestycki, Caffarelli and Nirenberg  (see p. 486 in \cite{BCN1}).\\

\begin{thm}\label{T:3spe}
Let $u\in C^2(\overline{\Sigma_{2b}})$ be a nonnegative  solution to
\begin{equation}\label{E:Pstrip}
\begin{cases}
-\Delta u=f(u), & \text{ in } \Sigma_{2b}\\
\quad u \geqslant 0, & \text{ in } \Sigma_{2b}\\
\quad u=0,\,\, &\text{ on } \partial\Sigma_{2b} \,
\end{cases}
\end{equation}

with $f$ locally Lipschitz continuous on $[0\,, +\infty)$ \emph {(no restriction on the sign of $f(0)).$} Then

\item[$(i)$] if $ f(0) <0$, either $u$ is positive on $\Sigma_{2b}$, symmetric about $\{y=b\}$ with  $\partial_y u >0$ in $\Sigma_{b}$, or $u$ is one-dimensional and periodic. In this case $2b$ is necessarily a multiple of the period of $u$.

\item[$(ii)$]  if $f(0) \ge 0$, either $u$ vanishes identically or it is positive on $\Sigma_{2b}$, symmetric about $\{y=b\}$ with  $\partial_y u >0$ in $\Sigma_{b}$.
\end{thm}

\begin{rem}
Note that the  theorem above also applies when $f(0)\ge0$ and it is \emph{new} even in this case.
\end{rem}






Before proceeding further, let us briefly recall the main difficulties that one has to face when working in our general framework.

When $f(0)<0$, nonnegative solutions are natural and must be taken into account. Indeed, in this case, nontrivial, nonnegative solutions (vanishing somewhere) can exist and sometimes, they are the only nonnegative solutions of the considered problem (this is the case when $f(u) =u-1$, for instance). These phenomena are strongly related to the absence of both the \emph{strong maximum principle} and the \emph{Hopf's lemma}. Hence a deeper and different analysis with respect to the case $f(0)\geqslant 0$ is needed. Further difficulties in the analysis are added by the fact that the solution $u$ is \emph {not} assumed to be bounded and $f$ is \emph {merely locally Lipschitz continuous}. Indeed, in the study of the qualitative properties of the solutions to semilinear problems in unbounded domains, it is always assumed some \emph {a priori} bound on $u$ and/or the \emph{global Lipschitz character} of $f$. These properties ensure, for instance, the possibilty to use elliptic estima
 tes to study the asymptotic behaviour $u$, by means of the translation invariance of the considered problem  and/or to use some comparison principles on unbounded cylindrical domains having small cross section.  The lack of those properties in our general framework will require a different approach to the problem.

To obtain our results we shall use a \emph{rotating plane method} (inspired by \cite {BCN2}, and especially by \cite{DS3}) combined with the \emph{unique continuation principle} (see \cite{FVb} and the references therein). These tools are described and developed in sections 2 and 3.

To continue the description of our results, we denote by $p\,:=\,(x,y)$ a general point in the plane and, for a nonnegative solution $u$ of \eqref{E:P}, we say that $u$ satisfies the property $(\mathcal P_\mu)$ if {\it there exists a real number $ \mu>0$ and a point  $p\in\{y=\mu\}$ such that $u(p)\neq 0$. }

Equivalently :
\begin{equation}\nonumber
(\mathcal P_\mu)\qquad\text{holds if}\qquad\{y=\mu\}\cap\{u\neq0\}\neq\emptyset\,.
\end{equation}

Since $u$ cannot be identically zero by the assumption $ f(0) <0$, we see that the above property $(\mathcal P_\mu)$ is satisfied for some $ \mu>0$.

We shall prove in Theorem \ref{dfbjkrgbkcvdvcgdcg} that the set
\begin{equation}\label{LAMBDA}
\Lambda^*=\Lambda^*(u)\,:=\,\{\lambda>0\,\,:\,\,(\mathcal P_\mu)\quad\text{holds for every} \,\, 0<\mu\leq\lambda\}
\end{equation}

is not empty (in any dimension $N\ge 2$). Therefore we have
\begin{equation}\label{lambda}
\lambda^*=\lambda^*(u)\,:=\, \sup \Lambda^*\, \in (0, +\infty].
\end{equation}

Note that, by a continuity argument, if $\lambda^*$ is finite, we easily get that  $\{y=\lambda^*\}\subseteq \{u=0\}$.

\noindent Recalling the notation
$\Sigma_\lambda:= \{(x,y) \in \R^2 \, |\, 0<y<\lambda\}$,
we have the following

\begin{thm}\label{T:1}
Assume that $f$ is locally Lipschitz continuous on $[0\,, +\infty)$ with $f(0) <0$.
Let $u\in C^2(\overline{\mathbb{R}^2_+})$ be a nonnegative solution to \eqref{E:P}. Then

\item[$(i)$] if $\lambda^*=+\infty$, $u$ is positive and strictly monotone increasing in the y direction, with
\begin{equation}\nonumber
u >0, \qquad \partial_y u > 0\qquad \text{in}\quad \mathbb{R}^2_+.
\end{equation}

\item[$(ii)$] If $\lambda^* < + \infty$, $u$ is one-dimensional and periodic, i.e.
\[
u(x,y)\,=\,u_0(y) \qquad \forall \, (x,y) \in \R^2_+
\]

where $ u_0\in C^2(\R, [0,\infty))$ is periodic of period $\lambda^*$. Moreover, $u_0$ is the unique solution of

\begin{equation}\label{oneD}
- u_0'' = f(u_0) \quad\text{in}\,\,[0,\infty)\qquad
u_0'(0)=u_0(0) = 0 =u_0'(\lambda^*)=u_0(\lambda^*)\, .
\end{equation}

Also, $u$ is symmetric with respect to $\{y=\frac{\lambda^*}{2}\}$ with $\partial_y u > 0$ in $\Sigma_{\frac{\lambda^*}{2}}$.
\end{thm}

We shall provide two different proofs of the above theorem. Note that Theorem \ref{casteorem} is an immediate consequence of Theorem \ref{T:1}. Both of them provide a complete classification of the solutions to problem \eqref{E:P} and they significantly extend the results of \cite{BCN2}, where it is always assumed that $f$ is \emph{globally Lipschitz continuous} and/or the \emph {solution $u$ is positive and bounded}, and the partial results obtained in \cite{Dancer2}, which hold for \emph {bounded solutions and $f \in C^1$. }



Next we have two results concerning the one-dimensional symmetry for solutions to \eqref{E:P}. \\

\begin{thm}\label{casteorembis}
Assume that $f$ is locally Lipschitz continuous on $[0\,, +\infty)$.
Let $u\in C^2(\overline{\mathbb{R}^2_+})$ be a nonnegative solution to \eqref{E:P} with
$|\nabla \,u|\in L^\infty(\mathbb{R}^2_+)$.
Then $u$ is one-dimensional, i.e.
\[
u(x,y)\,=\,u(y)\, \qquad \forall \, (x,y) \in \R^2_+.
\]
\end{thm}

The above theorem recovers and improves upon a result of \cite{FV}, where only positive solutions were considered (cf. also \cite{BCN2}). \\

\begin{thm}\label{xxx5}
Assume that $f \in C^1([0\,, +\infty))$ with $f(0) <0$ such that
\[
f'(t)\geqslant c>0 
\quad\text{for any}\quad t\in (0,\infty)
\]

and let $u\in C^2(\overline{\mathbb{R}^2_+})$ be a nonnegative solution to \eqref{E:P}.

Then $u$ is one-dimensional and periodic i.e.
\[
u(x,y)\,=\,u_0(y) \qquad \forall \, (x,y) \in \R^2_+
\]
with $u_0$ as in Theorem \ref{T:1}. In particular, there are no positive solutions to \eqref{E:P}.\\
\end{thm}
Note that Theorem \ref{Txxx5} is obtained by setting $f(u) = u-1$ in the previous theorem.\\

We can now turn to the case of the strips $\Sigma_{2b}$, $b>0$. Precisely, we consider the following problem
\begin{equation}\label{E:Pstrip}
\begin{cases}
-\Delta u=f(u), & \text{ in } \Sigma_{2b}\\
\quad u \geqslant 0, & \text{ in } \Sigma_{2b}\\
\quad u=0,\,\, &\text{ on } \partial\Sigma_{2b} \,.
\end{cases}
\end{equation}
We let $\Lambda^*$ be defined by \eqref{LAMBDA}, considering there values of $\lambda$ such that $0<\lambda<2b$, thus $\lambda^* \in (0, 2b]$.

We shall prove

\begin{thm}\label{T:3}
Assume that $f$ is locally Lipschitz continuous on $[0\,, +\infty)$ with $f(0)<0. $ Let $u\in C^2(\overline{\Sigma_{2b}})$ be a nonnegative  solution to \eqref{E:Pstrip}.  Then,

\item[$(i)$] if $\lambda^*=2b$, it follows that $u$ is positive in $\Sigma_{2b}$ with
\begin{equation}\nonumber
\partial_y u > 0\qquad \text{in}\quad \Sigma_b.
\end{equation}
Furthermore $u$ is symmetric with respect to $\{y=b\}$ i.e., $u(x,y)\,=\,u(x,2b-y)$ for any $0\leqslant y\leqslant 2b$.\\

\item[$(ii)$] If $\lambda^* < 2b$, $u$ is one-dimensional and periodic, i.e.

\[
u(x,y)\,=\,u_0(y) \qquad \forall \, (x,y) \in \Sigma_{2b}
\]

where $ u_0\in C^2(\R, [0,\infty))$ is periodic of period $\lambda^*$. Moreover, $u_0$ is the unique solution of

\begin{equation}\label{oneDstrip}
- u_0'' = f(u_0) \quad\text{in}\,\,[0,2b]\qquad
u_0(0)=u_0'(0)=0.
\end{equation}

Also, $u$ is symmetric with respect to $\{y=\frac{\lambda^*}{2}\}$ with $\partial_y u > 0$ in $\Sigma_{\frac{\lambda^*}{2}}$. Finally, $2b$ is necessarily a multiple of the period $\lambda^*$.
\end{thm}


The techniques used to prove Theorem \ref{T:3} also cover the case $ f(0) \ge 0$. Since the result appears to be \emph{new} even in this case, we explicitely state it in the next Theorem \ref{T:3+}.

\begin{thm}\label{T:3+}
Assume that $f$ is locally Lipschitz continuous on $[0\,, +\infty)$ with $f(0)\ge 0. $ Let $u\in C^2(\overline{\Sigma_{2b}})$ be a nonnegative  solution to \eqref{E:Pstrip}.

Then, either $u$ vanishes identically or it is positive on $\Sigma_{2b}$, symmetric about $\{y=b\}$ with  $\partial_y u >0$ in $\Sigma_{b}$.

 \end{thm}

Note that Theorem \ref{T:3spe} is a consequence of the combination of Theorem \ref{T:3} and Theorem \ref{T:3+}.

The next result concerns the qualitative properties of nonnegative solutions on coercive epigraphs. It holds true in every dimension $N \ge 2$. Let us recall that a domain $\Omega \subset \R^N$ is a \emph{smooth coercive epigraph} if, up to a rotation of the space, there exists $g \in C^{2}(\R^{N-1},\R)$ such that $ \Omega : = \{ \, x = (x',x_N) \in \R^{N-1} \times \R \, : \, x_N > g(x') \, \}$ and $ \lim_{\vert x' \vert \to + \infty} g(x') = +\infty.$ \

\begin{thm}\label{T:Epi}
Let $N\ge 2$ and let $\Omega \subset \R^N $ denote a smooth coercive epigraph. Assume that $f$ is locally Lipschitz continuous on $[0\,, +\infty)$ with $f(0)<0. $ Let $u\in C^2(\overline{\Omega})$ be a nonnegative  solution to
\begin{equation}\label{E:Epi}
\begin{cases}
-\Delta u=f(u), & \text{ in } \Omega\\
\quad u \geqslant 0, & \text{ in } \Omega\\
\quad u=0,\,\, &\text{ on } \partial\Omega\,.
\end{cases}
\end{equation}

Then, $u$ is positive and strictly monotone increasing in the $x_N$ direction, with
\begin{equation}\nonumber
u >0, \qquad \partial_{x_N} u > 0\qquad \text{in}\quad \Omega.
\end{equation}
\end{thm}

Previous results in this case have been obtained in \cite{BCN3,EL} under the condition that $u$ \emph{is a positive solution} (cf. also \cite{Fa2}).

Further results concerning the higher dimensional case are provided in the last section.

In this work we focused on the case $f(0)<0$, where very few results were available. For the more classical case $f(0) \ge 0$, $f$ globally Lipschitz continuous and/or $u$ bounded, we refer to \cite{BCN1,BCN2,BCN3,BCN5,CLZ,CMS,Dancer,EL,Fa,Fa2,FSV,FV,GiSp} and the references therein.

The paper is organized as follows: in Section \ref{prelimjkdfbjsk} we state and prove some preliminary results needed for the application of 
the rotating plane technique. In Section \ref{sjjkdkgbnnbnbnbnbn} we give the first proof of Theorem~\ref{T:1}. Here we exploit the unique continuation principle only to prove the last assertion of the statement. We provide the second proof of Theorem~\ref{T:1} in Section \ref{kvbvvcvcvcvcvcvv} exploiting there the unique continuation principle to start the rotating plane procedure. Section \ref{fdfdfdfdfdffdfdfdfdfd} is devoted to the proof of Theorems \ref{casteorem}-\ref{T:3spe} and Theorems \ref{casteorembis}-\ref{T:3+}. The results in higher dimensions are treated in Section \ref{higer}.

\section{Preliminary results}\label{prelimjkdfbjsk}


In this section we assume that $f$ is \emph {merely locally Lipschitz continous}. No restrictions are imposed on the sign of $f(0)$.

In the proof of our main result we will exploit 
a rotating plane technique. This will be strongly based on the use of weak and strong maximum principles, see e.g. \cite{GT,PSB}.
Since we are not assuming that the solution is globally bounded and since we are not assuming that the nonlinearity $f$ is globally Lipschitz continuous, we need the following version of the weak comparison principle in domains of small measure.

\begin{prop}[Weak Comparison Principle in small domains]\label{maxpri}
Assume $ N\ge 2$. Let us consider a bounded domain $D\subset\R^N$ and $u,v\in C^2(\overline D)$ such that
\begin{equation}\label{eqdifferenza}
-\Delta u-f(u)\leq -\Delta v-f(v)\,\qquad\text{in}\quad D\,.
\end{equation}
Then there exists $\vartheta\,=\,\vartheta(D,u,v,f)>0$ such that, for any domain $D'\subset D$
with $u\leq v$ on $\partial\,D'$
and $\mathcal L (D')\leq \vartheta$, it follows
\[
u\leq v\qquad\text{in}\,\,D'\,.
\]
\end{prop}
\begin{proof}
We use $(u-v)^+\in H^1_0(D')$ as test function in the weak formulation of \eqref{eqdifferenza} and get
\begin{equation}\nonumber
\begin{split}
\int_{D'} \big|\nabla (u-v)^+\big|^2\,dx&\leq \int_{D'} \,\frac{f(u)-f(v)}{(u-v)}((u-v)^+)^2\,dx\\
&\leqslant C(D,u,v,f)\int_{D'} \,((u-v)^+)^2\,dx
\end{split}
\end{equation}

where the positive constant $C(D,u,v,f)$ can be determined exploiting the fact that $u,v$ are bounded on $ {\overline D}$ and $f$ is locally Lipschitz continuous on $[0\,, +\infty)$.

An application of Poincar\'e inequality gives

\begin{equation}\nonumber
\int_{D'} \big|\nabla (u-v)^+\big|^2\,dx \leqslant C(D,u,v,f)(C_N (\mathcal L (D'))^{\frac{2}{N}}) \int_{D'} \big|\nabla (u-v)^+\big|^2\,dx\,,
\end{equation}

where $C_N>0$ is a constant depending only on the euclidean dimension $N$.

\noindent For $\mathcal L (D')$ small such that
$ C(D,u,v,f)(C_N (\mathcal L (D'))^{\frac{2}{N}})<1$  we get that $(u-v)^+\equiv 0$ and the thesis.
\end{proof}

Now we focus on the two-dimensional case and fix some notations.
Given $x_0 \in \R$, $s>0$ and $\theta\in(0\,,\,\frac{\pi}{2})$, let $L_{x_0,s,\theta}$ be the line, with slope $\tan(\theta)$, passing through $(x_0,s)$. Also, let
$V_\theta$ be the vector orthogonal to $L_{x_0,s,\theta}$ such that $(V_\theta,e_2) >0$ and $\|V_\theta\|=1$.

We denote by $$\mathcal{T}_{x_0,s,\theta}$$ the (open) triangle delimited by $L_{x_0,s,\theta}$, $\{y=0\}$ and $\{x=x_0\}$.
We also define
$$
u_{x_0,s,\theta}(x)=u(T_{x_0,s,\theta}(x)), \quad x \in \mathcal{T}_{x_0,s,\theta}
$$
where $T_{x_0,s,\theta}(x)$ is the point symmetric to $x$, w.r.t. $L_{x_0,s,\theta}$, and
\begin{equation}\label{fskhkhgfjhgjhfj}
w_{x_0,s,\theta}=u-u_{x_0,s,\theta}\,.
\end{equation}

It is immediate to see that
$u_{x_0,s,\theta}$ still fulfills $-\Delta u_{x_0,s,\theta}=f(u_{x_0,s,\theta})$ and
\begin{equation}\label{differenzatsrejkgfjk}
-\Delta  w_{x_0,s,\theta}= c_{x_0,s,\theta} w_{x_0,s,\theta}
\end{equation}

on the triangle $\mathcal{T}_{x_0,s,\theta}$, where we have set

\begin{equation}\label{differenzatsrejkgfjk}
c_{x_0,s,\theta}(x) : = \begin{cases}
\frac{f(u(x))-f(u_{x_0,s,\theta}(x))}{u(x)-u_{x_0,s,\theta}(x)}  & \text{ if }  w_{x_0,s,\theta} \neq 0 \\
\quad 0 & \text{ if }  w_{x_0,s,\theta} = 0.
\end{cases}
\end{equation}

Note that $\vert c_{x_0,s,\theta} \vert \le C(\mathcal{T}_{x_0,s,\theta},u,f)$ on the triangle $\mathcal{T}_{x_0,s,\theta}$, where  $C(\mathcal{T}_{x_0,s,\theta},u,f)$ is a positive constant which can be determined by exploiting the fact that $u$ and $u_{x_0,s,\theta}$ are bounded on $\overline {\mathcal{T}_{x_0,s,\theta}}$ and $f$ is locally Lipschitz continuous on $[0\,, +\infty)$.



In what follows we shall make repeated use of a refined version of the \emph{moving plane technique} \cite{S} (see also \cite{BN,GNN}). Actually we will exploit a
\emph{rotating plane technique} and a \emph{sliding plane technique}  developed in \cite{DS3}.

Let us give the following definition

\begin{definition}\label{defcondition} Given $x_0, s$ and $\theta $ as above, we say that the condition  $(\mathcal H\mathcal T_{x_0,s,\theta})$ holds in the triangle $\mathcal T_{x_0,s,\theta}$ if
\begin{center}

$\quad w_{x_0,s,\theta}< 0 \quad $ in $\quad \mathcal{T}_{x_0,s,\theta}$,

$w_{x_0,s,\theta}\leqslant 0 \quad $ on $ \quad \partial(\mathcal{T}_{x_0,s,\theta}) \quad $ and

$ w_{x_0,s,\theta}$ is not identically zero on $\partial(\mathcal{T}_{x_0,s,\theta})$,
\end{center}

with $w_{x_0,s,\theta}$ defined in \eqref{fskhkhgfjhgjhfj}.
\end{definition}

We have the following

\begin{lem}[Small Perturbations]\label{smallperturbations}
Let  $(x_0,s,\theta)$  and $\mathcal{T}_{x_0,s,\theta}$ be  as above and assume that
$(\mathcal H\mathcal T_{x_0,s,\theta})$ holds. Then there exists $\bar\mu = \bar\mu(x_0, s, \theta)>0$ such that
\begin{equation}\label{claimsmallpert}
\begin{cases}
\vert \theta - \theta' \vert + \vert s-s'\vert < \bar\mu, &\\
w_{x_0,s',\theta'}\leqslant 0 \quad \text {on} \,\, \partial(\mathcal{T}_{x_0,s',\theta'}), \hskip4truecm \Longrightarrow \quad (\mathcal H\mathcal T_{x_0,s',\theta'}) \,\, \text {holds}.&\\
w_{x_0,s',\theta'} \quad \text {is not identically zero on} \,\, \partial(\mathcal{T}_{x_0,s',\theta'}) &\\
\end{cases}
\end{equation}
\end{lem}


\begin{proof}
In order to exploit Proposition \ref{maxpri} let us fix
a bounded domain $D\subset {\mathbb{R}^2_+}$ such that $\mathcal{T}_{x_0,s,\theta}\subset D$ for all $ s' \in (0,s+1) $ and $ \theta' \in (\frac{\theta}{2}, \frac{\pi}{2})$. Now pick a small $\epsilon = \epsilon(\theta,s)>0$ such that
$ \mathcal{L} (\mathcal{T}_{x_0,s+\epsilon,\theta- \epsilon} \setminus \mathcal{T}_{x_0,s-\epsilon,\theta+ \epsilon}) < \frac{\vartheta}{10}$ and then a compact set $K\subset \mathcal{T}_{x_0,s-\epsilon,\theta+ \epsilon}$ such that $ \mathcal{L} (\mathcal{T}_{x_0,s-\epsilon,\theta+ \epsilon} \setminus K) < \frac{\vartheta}{10}$,
where $\vartheta$ is given by Proposition \ref{maxpri}. Therefore,
for all $(s', \theta')$ satisfying $ \vert \theta - \theta' \vert + \vert s-s'\vert < \epsilon,$ we have $ \mathcal{L} (\mathcal{T}_{x_0,s',\theta'} \setminus K) < \frac{\vartheta}{5}$. Also, since by assumption $w_{x_0,s,\theta}< 0$ in $\mathcal{T}_{x_0,s,\theta}$  we get $w_{x_0,s,\theta} \leqslant\rho < 0$ on the compact set $K$.  Therefore, we can find $ \bar\mu \in (0, \epsilon)$ such that, for all $(s', \theta')$ satisfying $ \vert \theta - \theta' \vert + \vert s-s'\vert < \bar\mu,$ we have $w_{x_0,s',\theta'} \leqslant \frac{\rho}{2} < 0$ on the compact set $K$ and
$ \mathcal{L} (\mathcal{T}_{x_0,s',\theta'} \setminus K) < \frac{\vartheta}{5}$.
Since $w_{x_0,s',\theta'}\leqslant  0$ on $\partial \big(\mathcal{T}_{x_0,s',\theta'}\setminus K\big)$, we can apply Proposition \ref{maxpri} to get that
\begin{center}
$w_{x_0,s',\theta'}\leqslant  0 \quad $  in $\quad \mathcal{T}_{x_0,s',\theta'}\setminus K$
\end{center}
and therefore in the triangle $\mathcal{T}_{x_0,s',\theta'}$. Also by the strong comparison principle, we  get
 $$w_{x_0,s',\theta'}<  0\qquad \text{in}\quad\mathcal{T}_{x_0,s',\theta'}$$
  and the proof is completed.
\end{proof}


Let us now show that, from the fact that we can make small translations and rotations of  $\mathcal{T}_{x_0,s,\theta}$ towards $\mathcal{T}_{x_0,s',\theta'}$,
if $(s',\theta ')\approx (s,\theta)$, then \emph{we can also make larger translations and rotations}. We have the following

\begin{lem}[The sliding-rotating technique]\label{largeper}
Let $(x_0,s,\theta )$ be as above ad assume that $(\mathcal H\mathcal T_{x_0,s,\theta})$ holds.
 Let $(\hat{s},\hat{\theta})$ be fixed 
and assume that there exists a continuous function $g(t)=(s(t),\theta (t))\, :\,[0\,,\,1] \rightarrow (0, +\infty) \times (0\,,\,\frac{\pi}{2})$, such that
$g(0)=(s,\theta)$ and $g(1)=(\hat{s},\hat{\theta})$.
Assume that
\begin{equation}\label{tttt}
 w_{x_0,s(t),\theta(t)}\leqslant 0\qquad \text{on}\quad\partial(\mathcal{T}_{x_0,s(t),\theta(t)})\qquad
 \text{for every} \quad t\in [0,1)
 \end{equation}

and that
$w_{x_0,s(t),\theta(t)}$ is not identically zero on $\partial(\mathcal{T}_{x_0,s(t),\theta(t)})$  for every $t\in [0,1)$.\\

Then
\begin{center}
$(\mathcal H\mathcal T_{x_0,\hat s,\hat \theta})$ holds.
\end{center}
\end{lem}

\begin{proof}
By the assumptions and exploiting Lemma \ref{smallperturbations} we  obtain the existence of  $\tilde{t}>0$ small such that, for  $0\leqslant t\leqslant \tilde{t}$, $(\mathcal H\mathcal T_{x_0,s(t),\theta(t)})$ holds.


We now set

\begin{center}
$\overline{T}\equiv \,\, \{\tilde{t}\in [0,1]\,\, s.t.\, (\mathcal H\mathcal T_{x_0,s(t),\theta(t)})\,\,\text{  holds for any}\quad 0 \leqslant t\leqslant\tilde{t}\}$
\end{center}
 and $$\bar{t}=\sup\,\overline{T}\,.$$
We claim that actually $\bar{t}=1$. To prove this, assume  $\bar{t}<1$ and note that in this case we have
$$
w_{x_0,s(\bar{t}),\theta(\bar{t})}\le  0 \quad \text{in}\quad  \mathcal{T}_{x_0,s(\bar{t}),\theta(\bar{t})}\,
$$
$$
w_{x_0,s(\bar{t}),\theta(\bar{t})} \le 0 \quad on \quad \partial(\mathcal{T}_{x_0,s(\bar{t}),\theta(\bar{t})})
$$

by continuity, and that $w_{x_0,s(\bar{t}),\theta(\bar{t})}$ is not identically zero on $\partial(\mathcal{T}_{x_0,s(\bar{t}),\theta(\bar{t})})$  by assumption.

Hence, by the strong maximum principle, we see that
$$
w_{x_0,s(\bar{t}),\theta(\bar{t})}< 0 \quad \text{in}\quad  \mathcal{T}_{x_0,s(\bar{t}),\theta(\bar{t})}\,.
$$

Therefore $(\mathcal H\mathcal T_{x_0,s(\bar{t}),\theta(\bar{t})})$ holds and using once again Lemma \ref{smallperturbations}, we can find a sufficiently small $\varepsilon >0$ so that  $(\mathcal H\mathcal T_{x_0, s(t),\theta(t)})$ holds for any
$0\leqslant t \leqslant \bar{t}+\varepsilon$, which contradicts the definition of $\bar{t}$.
\end{proof}

\section{First Proof of Theorem~\ref{T:1}}\label{sjjkdkgbnnbnbnbnbn}

Given \emph {any} $x_0 \in \R$, let us set
\begin{equation}\label{QH}
Q_{h}(x_0)=\{(x,y)\,:\, |x-x_0| \leqslant h, 0 \leqslant y \leqslant 2h\}\,.
\end{equation}
Since $\partial_{xx}u(x_0,0) = 0$, we have that
\begin{equation}\nonumber
-\partial_{yy}u(x_0,0)\,=\,-\Delta u (x_0,0)=f(u(x_0,0))=f(0)<0\,,
\end{equation}
by $(f_2)$. Recalling that $u\in C^2(\overline{\mathbb{R}^2_+})$, we conclude that we can take $\bar h>0$ small such that
\begin{equation}\nonumber
 \partial_{yy}u>0\qquad \text{in }\quad Q_{\bar h}(x_0).
\end{equation}

Exploiting again the fact that $u\in C^2(\overline{\mathbb{R}^2_+})$, we  can consequently find $\bar \theta=\bar \theta(\bar h) \in (0, \frac{\pi}{2}) $ such that
\begin{equation}\label{monteta}
\frac{\partial}{\partial V_\theta}  \left(\frac{\partial\,\,u}{\partial V_\theta} \right)>0\qquad \text{in }\quad Q_{\bar h}(x_0)\qquad\text{for}\,\,\,\, - \bar\theta \leq\theta\leq\bar\theta\,.
\end{equation}

Also, since we assumed that $u$ is nonnegative in $\mathbb{R}^2_+$, it follows that
\begin{equation}\label{sbvdvjkdbvdvbjkd}
\frac{\partial\,\,u}{\partial V_\theta}(x,0) \geqslant 0\qquad\qquad\text{for any}\,\,\,\, -\bar \theta \leq\theta\leq\bar\theta
 \quad \text{and for any}\,\,\,x\in\R\,.
\end{equation}

By combining \eqref{monteta} and \eqref{sbvdvjkdbvdvbjkd}, we deduce the strict monotonicity of $u$ in the $V_\theta$-direction, for every $ x \in Q_{\bar h}(x_0) \cap \R^2_+$ and every
$ \theta \in [-\bar\theta , \bar \theta] $.


>From the above analysis, we find the existence of (possible very small)

\begin{equation}\label{mons}
\bar s\,=\,\bar s(\bar \theta)>0\,,
\end{equation}
such that, for any $0< s\leqslant \bar{s}$ :

\begin{itemize}
 \item[$i)$] both the triangle $\mathcal{T}_{x_0,s,\bar\theta}$ and its reflection w.r.t. $L_{x_0,s,\bar\theta}$ are contained in $Q_{\bar h}(x_0)$ (as well as their reflections w.r.t. the axis $ \{\, x = x_0 \, \}$),
\item[$ii)$] both the segment $\{\, (x_0,y) \, : \,  0 \le y \le s  \, \}$ and its reflection w.r.t. $L_{x_0,s,\theta}$ are contained in $Q_{\bar h}(x_0)$ for every $\theta \in (0, \bar\theta]$,
\item[$iii)$] $u < u_{x_0,s,\bar\theta}$ in $\mathcal{T}_{x_0,s,\bar\theta}$,
\item[$iv)$]  $u \leqslant u_{x_0,s,\theta}$ on $\partial(\mathcal{T}_{x_0,s,\theta})$ for every $\theta \in (0, \bar\theta]$,
\item[$v)$]  $u < u_{x_0,s,\theta}$ on the set $ \{\, (x_0,y) \, : \,  0 < y < s  \, \}$, for every $\theta \in (0, \bar\theta]$.
\end{itemize}





Note that, from $iii)-iv)$, we have that
\begin{equation}\label{Hinizio}
\forall \, s \in (0, \bar s), \qquad \mathcal (\mathcal H\mathcal T_{x_0,s,\bar\theta}) \qquad holds.
\end{equation}

\medskip

\noindent Next we prove a result that allows to start the moving plane procedure:

\begin{lem}[Monotonicity near the boundary]\label{starting}
There exists $\hat \lambda>0$
such that, for any $0<\lambda\leq \hat\lambda$, we have
\begin{equation}\label{sdhsshkshkdvjbvj}
u<u_\lambda\qquad\text{in}\quad \Sigma_\lambda\,.
\end{equation}
Furthermore
\begin{equation}\label{monotHAT}
\partial_y u >0\qquad\text{in}\quad \Sigma_{\hat\lambda}\,.
\end{equation}

\end{lem}

\begin{proof}
Let $\bar\theta$ given by \eqref{monteta} and $\, \bar s=\bar s(\bar\theta)$
 as in \eqref{mons}. We showed  that, for any $0<s<\bar s$,
$(\mathcal H\mathcal T_{x_0,s,\bar\theta})$ holds.

We use now Lemma \ref{largeper} as follows : for any fixed $s \in (0, \bar s)$ and $ \theta' \in (0, \bar\theta) $
we consider the rotation
\[
g(t)\,=\,(s(t),\theta (t))\,:=\,(s\,,\,t\theta'+(1-t)\bar\theta)\qquad\quad t\in[0\,,\,1]\,.
\]
Recalling that $(\mathcal H\mathcal T_{x_0,s,\bar\theta})$ holds by \eqref{Hinizio}, we deduce  that also
$(\mathcal H\mathcal T_{x_0,s,\theta'})$ holds.
Therefore, by the fact that $0<\theta'<\bar\theta$ is arbitrary and by a continuity argument, we pass to the limit for $\theta'\rightarrow 0$ and get
\begin{center}
 $u(x,y) \le  u_s(x,y)$ in $\Sigma_{s}\cap \{x\leqslant x_0\}$ for  $0<s<\bar s$.
 \end{center}


The invariance of the considered problem w.r.t. the axis $\{\, x = x_0 \,\} $ enables us to use the same argument to treat the case of negative $\theta$, yielding

\begin{center}
 $u(x,y) \le u_s(x,y)$ in $\Sigma_{s}\cap \{x\geqslant x_0\}$ for  $0<s<\bar s$,
 \end{center}
possibly reducing $\bar s$.

Thus $u(x,y) \le u_s(x,y)$ in $\Sigma_{s}$ for every $ s \in (0,\bar s)$.  The desired conclusion \eqref{sdhsshkshkdvjbvj} then follows by taking $\hat\lambda$ such that $  0 < \hat\lambda < \min \{\bar s, \frac{\lambda^*}{2}\}$. Here we have used in a crucial way that the property $(\mathcal P)_{\lambda}$ holds for every $ \lambda \in (0, \hat\lambda]$, so that the case $u\equiv u_{\lambda}$ in $\Sigma_{\lambda}$ is not possible.

Moreover, by the Hopf's Lemma, for every $ \lambda \in (0, \hat\lambda]$ and every $x \in \R$, we get
\begin{equation}
2\partial_y u(x, \lambda)=\frac{\partial (u-u_{\lambda})}{\partial y}(x, \lambda)>0\,.
\end{equation}
The latter proves \eqref{monotHAT}.

\end{proof}

To proceed further we need
\emph{some notations}: in the case $\lambda^*=\infty$ we set
$$\Lambda=\{\lambda>0\, :\,u<u_{\lambda '}\quad\text{in}\,\,\Sigma_{\lambda'}\,\,\,\,\forall \lambda'<\lambda\}\,.$$
If $\lambda^*$ is finite we use the same notation but considering values of $\lambda$ such that $0<\lambda<\lambda^*/2$, namely
$$\Lambda=\{\lambda<\frac{\lambda^*}{2}\, :\,u<u_{\lambda '}\quad\text{in}\,\,\Sigma_{\lambda'}\,\,\,\,\forall \lambda'<\lambda\}\,.$$
By Lemma \ref{starting} we know that $\Lambda$ is not empty and we can define
\begin{equation}\label{hhhhhhhhhh}
\bar{\lambda}=\sup\,\,\Lambda\,.
\end{equation}
\noindent The proof of the theorem will be  done if we show that $\bar\lambda=+\infty$, when $ \lambda^* = \infty$ (resp. $\bar\lambda=\frac{\lambda^*}{2}$, when
$\lambda^*$ is finite).
Therefore we argue by contradiction and assume that $\bar\lambda<+\infty$, when $ \lambda^* = \infty$ (resp. $\bar\lambda<\frac{\lambda^*}{2}$, when $\lambda^*$ is finite).\

First, as above, we deduce  that $u$ is strictly monotone increasing in the $e_2$-direction in $\Sigma_{\bar\lambda}$, with
\begin{equation}\label{monotogsfdadda}
\partial_y u \,>\,0\qquad\text{in}\quad\Sigma_{\bar{\lambda}}\,.
\end{equation}

To proceed further we need to prove the following

\begin{lem}\label{caffbe} Let $\lambda^*$ and $\bar\lambda$ be as above. Assume that there is a point  $x_0\in\mathbb{R}$ satisfying $u(x_0,2\bar\lambda)>0$. Then there exists $\bar \delta >0$ such that:
for any $-\bar\delta\leqslant \theta\leqslant\bar\delta $ and for any
$0< \lambda\leqslant\bar{\lambda}+\bar\delta$, we have
 $$u(x_0,y)<u_{x_0,\lambda,\theta}(x_0,y)\,,$$
for $0< y<\lambda$.
\end{lem}

\begin{proof}
First we note that
$\partial_y u (x_0,\bar{\lambda})>0$.
In fact, by construction $u< u_{\bar{\lambda}}$
in $\Sigma_{\bar{\lambda}}$. Therefore, by the Hopf's Lemma,  we have
\begin{equation}\label{fttfryyrhdhdhdhjsjsjs}
2\partial_y u (x_0,\bar{\lambda})=\frac{\partial (u-u_{\bar{\lambda}})}{\partial y}(x_0,\bar{\lambda})>0\,.
\end{equation}
We argue now by contradiction. If the lemma were false, we found a sequence of small $\delta _n\rightarrow 0$
and
$-\delta _n \leqslant\theta_n\leqslant \delta _n$,
 $0<\lambda_n\leqslant \bar{\lambda}+\delta _n$,
 $0< y_n<  \lambda_n$ with $$u(x_0,y_n)\geqslant u_{x_0,\lambda_n ,\theta_n}(x_0,y_n).$$

Possibly considering subsequences,  we  may and do assume that $\lambda_n\rightarrow \tilde{\lambda}\leqslant \bar{\lambda}$. Also
$y_n\rightarrow \tilde{y}$ for some $\tilde{y}\leqslant \tilde{\lambda}$. Considering the construction of $Q_{\bar h}(x_0)$ as above and in particular taking into account \eqref{monteta} and \eqref{sbvdvjkdbvdvbjkd}, we deduce that $\tilde\lambda>0$ and,
by continuity, it follows that $u(x_0,\tilde{y})\geqslant u_{\tilde{\lambda}}(x_0,\tilde{y})$.
Consequently
$y_n\rightarrow \tilde{\lambda}=\tilde{y}$, since  we know that $u<u_{\lambda '}$  in $\,\Sigma_{\lambda '}$ for any $\lambda'\leqslant \bar{\lambda}$ and $u(x_0,0)=0<u(x_0,2\bar\lambda)$.
By the mean value theorem since $u(x_0,y_n)\geqslant u_{x_0,\lambda_n ,\theta_n}(x_0,y_n)$, it follows
$$\frac{\partial u}{\partial V_{\theta_n}}(x_n,y_n)\leqslant 0$$
at some point $\xi_n\equiv (x_n,y_n)$ lying on  the line from $(x_0,y_n)$ to $T_{x_0,\lambda_n,\theta_n}(x_0,y_n)$, recalling that
the vector $V_{\theta_n}$ is orthogonal to the line $L_{x_0,\lambda_n,\theta_n}$.  Since $V_{\theta_n}\rightarrow e_2$ as $\theta_n\rightarrow 0$.\\
Taking the limit it follows
$$\partial_y u (x_0,\tilde{\lambda})\leqslant 0$$
which is impossible by \eqref{fttfryyrhdhdhdhjsjsjs} and \eqref{monotogsfdadda}.
\end{proof}

\begin{proof}[\underline{End of the first Proof of Theorem \ref{T:1}}]

Since we are assuming that $\bar\lambda<+\infty$, when $ \lambda^* = \infty$ (resp. $\bar\lambda<\frac{\lambda^*}{2}$, when $\lambda^*$ is finite), we can find $x_0 \in \R$ such
that $u(x_0,2\bar\lambda)>0$. Let $Q_{\bar h}(x_0)$ be constructed as above and pick $\bar\theta$ given by \eqref{monteta}. 

Let also $\bar \delta $ as in Lemma \ref{caffbe}. Then fix $\theta_0>0$ with $\theta_0\leqslant \bar \delta$ and $\theta_0\leqslant\bar\theta$. Let us set
\[
s_0\,:=\,s_0(\theta_0)\,,
\]
such that  the triangle $\mathcal T_{x_0,s_0,\theta_0}$ and its reflection w.r.t. $L_{x_0,s_0,\theta_0}$ is contained in $Q_{\bar h}(x_0)$ and consequently $(\mathcal H\mathcal T_{x_0,s_0,\theta_0})$ holds. It is convenient to assume that $s_0\leqslant \hat\lambda$  with $\hat\lambda$ as in Lemma \ref{starting}.
For any
\[
s_0<s\leqslant\bar\lambda+\bar\delta,\qquad 0<\theta<\theta_0\,,
\]
we carry out the \emph{sliding-rotating technique} exploiting
Lemma \ref{largeper} with
\[
g(t)\,=\,(s(t),\theta (t))\,:=\,(ts+(1-t)s_0\,,\,t\theta+(1-t)\theta_0)\qquad\quad t\in[0\,,\,1]\,.
\]
By Lemma \ref{caffbe} we deduce that the boundary conditions required to apply Lemma \ref{largeper} are fulfilled and therefore, by Lemma \ref{largeper}, we get that $(\mathcal H\mathcal T_{x_0,s,\theta})$ holds.  We can now argue as in the proof of Lemma \ref{starting} and deduce that
 $u(x,y)<u_\lambda(x,y)$ in $\Sigma_{\lambda}$  for any $0<\lambda\leqslant \bar\lambda+\bar\delta$. This provides a contradiction unless $\bar\lambda=+\infty$
 (resp.
$\bar\lambda=\,\frac{\lambda^*}{2}$, if $\lambda^*$ is finite). Arguing e.g. as in the proof of Lemma \ref{starting}, we deduce
 \[
\partial_y u >0\qquad\text{in}\quad\mathbb{R}^2_+\qquad\,\,\,\text{if}\quad\lambda^*=+\infty\,,
 \]
while
 \begin{equation}\nonumber
\partial_y u > 0\qquad \text{in}\quad\Sigma _{\frac{\lambda^*}{2}}\qquad\text{if}\quad\lambda^*<+\infty\,.
\end{equation}
As a consequence of the monotonicity result, we deduce that $u$ is positive in $\mathbb{R}^2_+$ if $\lambda^*=+\infty$.\\

\noindent In  we assume that $\lambda^*$ is finite, we deduce by continuity that
\[
u\leq u_{\lambda^*/2} \qquad \text{in}\quad\Sigma_{\lambda^*/2}\,.
\]
By the strong comparison principle, we deduce that: either  $u< u_{\lambda^*/2}$ or $u\equiv u_{\lambda^*/2}$, in $\Sigma_{\lambda^*/2}$.
Note that, by the definition of $\lambda^*$, we have that $\{y=\lambda^*\}\subseteq \{u=0\}$, that also implies $\{y=\lambda^*\}\subseteq \{\nabla u=0\}$ since $u$ is nonnegative. If $u< u_{\lambda^*/2}$  in $\Sigma_{\lambda^*/2}$,
we get by the Hopf's boundary Lemma (see \cite{GT}) that $\partial_y (u_{\lambda^*/2}-u)>0$  on $\{y=0\}$. Since $\partial_y (u_{\lambda^*/2})=0$ on $\{y=0\}$ (by the fact that $\{y=\lambda^*\}\subseteq \{\nabla u=0\}$) this provides a contradiction with the fact that $u$ is nonnegative. Therefore it occurs $u\equiv u_{\lambda^*/2}$, in $\Sigma_{\lambda^*/2}$. \\

Note now that, since
$
\{y=\lambda^*\}\subseteq \{u=0\}\cap\{\nabla u=0\}\,,
$
by symmetry we deduce
\[
\{y=0\}\subseteq \{u=0\}\cap\{\nabla u=0\}\,.
\]
Therefore we deduce that $u$ is one-dimensional by the \emph{unique continuation principle} (see for instance Theorem 1 of \cite{FVb} and the references therein). Indeed, for every $t \in \R$, the function $u^t(x,y) : = u(x+t,y) $ is a nonnegative solution of \eqref{E:P} with $ u^t = \nabla u^t = 0$ on $ \partial \R^2_+$ and the unique continuation principle implies that $ u \equiv u^t $ on $\R^2_+$. This immediately gives that $u$ depends only on the variable $y$, i.e.,
\[
u(x,y)\,=\,u_0(y) \qquad \forall \, (x,y) \in \R^2_+
\]

where $u_0 \in C^2([0, +\infty)) $ is the unique solution of $u_0''+f(u_0)=0$ with $u_0'(0)=u_0(0)=0$.

The remaining part of the statement, namely the properties of $u_0$, follows by a simple ODE analysis.
\end{proof}

\section{Second Proof of Theorem~\ref{T:1}}\label{kvbvvcvcvcvcvcvv}

This proof makes use of the \emph{unique continuation principle} to start the moving plane procedure. \

\begin{proof}
[\underline{Second Proof of Theorem \ref{T:1}}]
\noindent If we assume that
\[
\nabla u (x,0)=0\qquad\text{for any}\quad x\in\R\,,
\]
then it follows that $u$ coincides with $u_0$ by the unique continuation principle, with $u_0$ as in the statement. A simple analysis of the associated ordinary differential equation shows in this case that $u_0$ is monotone increasing if $\lambda^*=\infty$, while $u_0$ is periodic when $\lambda^*$ is finite. Therefore the proof is done in this case and we reduce to consider the case:
\[
\text{there exists}\,\,x_0\in\R\quad\text{such that}\quad \nabla u(x_0,0)\neq 0\,.
\]
Necessarily in this case we have that $\partial_y u (x_0,0)>0$ since the case $\partial _y u (x_0,0)<0$ is not possible because $u$ is nonnegative.
Setting as above
\begin{equation}\label{QH}
Q_{h}(x_0)=\{(x,y)\,:\, |x-x_0| \leqslant h, 0 \leqslant y \leqslant 2h\}\,,
\end{equation}
recalling that $u\in C^2(\overline{\mathbb{R}^2_+})$, we can therefore fix
 $\bar h>0$ small such that
\begin{equation}\nonumber
 \partial_{y} u >0\qquad \text{in }\quad Q_{\bar h}(x_0).
\end{equation}
Exploiting again the fact that $u\in C^2(\overline{\mathbb{R}^2_+})$, we  can consequently consider $\bar \theta=\bar \theta(\bar h)$ small  such that
\begin{equation}\label{monteta2}
\frac{\partial\,\,u}{\partial V_\theta} >0\qquad \text{in }\quad Q_{\bar h}(x_0)\qquad\text{for}\,\,\,\, - \bar \theta \leq\theta\leq\bar\theta\,.
\end{equation}

>From the above \eqref{monteta2}, we immediately deduce the existence of $ \bar s$ as in \eqref{mons} and satisfying properties {\it i)-v)} at the beginning of the first proof of Theorem \ref{T:1} so that the moving plane procedure can be started.\

\noindent To proceed further note that, for $\bar\lambda$ defined as above, it occurs that
\[
u(x_0,2\bar\lambda)>0\,.
\]
In fact, if this is not the case, then we have $\nabla u(x_0,2\bar\lambda)=0$. Consequently
\[
\partial_y u (x_0,0) = \partial_y (u-u_{\bar\lambda})(x_0,0)\leqslant 0\,,
\]
while $\partial_y u (x_0,0)>0$ by construction.

Therefore Lemma \ref{caffbe} can be exploited
and, since the remaining part of the proof can be repeated verbatim, we omit it.

\end{proof}

\section{Proof of Theorems \ref{casteorem}-\ref{T:3spe} and Theorems \ref{casteorembis}-\ref{T:3+}}\label{fdfdfdfdfdffdfdfdfdfd}

\begin{proof}[\underline{Proof of Theorem \ref{casteorem}}]
By Theorem \ref{T:1} we have that, either $\lambda^*<+\infty$ and thus $u$ is one-dimensional and periodic or $\lambda^*=+\infty$ and $u$ satisfies $u>0$ and $\partial_y u >0$ in $\R^2_+$.

\end{proof}

\begin{proof}[\underline{Proof of Theorem \ref{casteorembis}}]
Let us first consider the case $f(0) <0$. By Theorem \ref{casteorem}, either $u$ is one-dimensional and periodic, and we are done, or $u>0$ and $\partial_y u >0$ in $\R^2_+$. Therefore, since we are assuming that $|\nabla u|$ is bounded, we are in position to apply Theorem 1.2  in \cite{FV} to conclude that $u$ is one-dimensional. When $ f(0) \ge 0,$ either $u$ is identically zero, or $u>0$ and $\partial_y u >0$ in $\R^2_+$ by Theorem 1.1 in \cite{DS3}. The desired conclusion then follows by applying once again Theorem 1.2  in \cite{FV}.

\end{proof}

\begin{proof}[\underline{Proof of Theorem \ref{xxx5}}]
Let us assume by contradiction that $\lambda^*=+\infty$ in Theorem \ref{T:1}, hence  $\partial_y u>0$ in $\R^2_+$. This implies that
\begin{equation}\nonumber
\int_{\R^2_+}\,|\nabla \varphi|^2\,-\,f'(u)\varphi^2\,dx\geqslant 0\qquad\text{for any}\quad\varphi\in C^\infty_c (\R^2_+)\,.
\end{equation}
Since we are assuming that $f'(t)\geqslant c>0$  for any $t\in (0,\infty)$, it follows that
\begin{equation}\label{jdhgkdhhgkshdvcvcvcvc}
\int_{\R^2_+}\,|\nabla \varphi|^2\,dx\geqslant c\int_{\R^2_+}\varphi^2\,dx \qquad\text{for any}\quad\varphi\in C^\infty_c  (\R^2_+)\,.
\end{equation}

Let us now consider an arbitrary  open ball $\Omega$ such that $\overline \Omega \subset \R^2_+$  and let $\lambda_1(\Omega)$ be the first eigenvalue of the Laplace operator in $\Omega$ under zero Dirichlet boundary conditions. By the variational characterization of $\lambda_1(\Omega)$ and by \eqref{jdhgkdhhgkshdvcvcvcvc},
it would follows that
\[
\lambda_1(\Omega)\geqslant c>0
\]
and this is clearly impossible since $\lambda_1(\Omega)$ approaches zero when $\Omega$ is chosen arbitrary large.
Therefore, necessarily, it occurs that $\lambda^*$ is finite, $u$ is one-dimensional and periodic and  $u$ cannot be positive.
\end{proof}

\begin{proof}[\underline{Proof of Theorem \ref{Txxx5}}]
\noindent By  Theorem \ref{xxx5}, $u$ is one-dimensional and periodic with profile $u_0$ satisfying \eqref{oneD} . A simple ODE analysis shows that $$u_0(y)=1-\cos y$$ and the result is proved.
\end{proof}

\begin{proof}[\underline{Proof of Theorem \ref{T:3}}]

The proof is analogous to the proof of Theorem \ref{T:1}. Let us only provide a few details.\\
\noindent Consider first the case  $\lambda^*=2b$. The moving plane procedure can be started as in Theorem \ref{T:1} exploiting Lemma \ref{starting}. Then we define $\bar\lambda$ as in \eqref{hhhhhhhhhh} ad we deduce that necessarily
$\bar\lambda=b$ arguing by contradiction exactly as in the proof of Theorem \ref{T:1} and exploiting the assumption $\lambda^*=2b$. Note that there is no need of changes in the proof of Lemma \ref{caffbe}. Therefore we deduce that  $u\leqslant u_\lambda$ in $\Sigma_\lambda$ for any $0<\lambda<b$ and, by continuity, we have
\[
u\leqslant u_b\qquad\text{in}\quad\Sigma_{b}\,.
\]
As a consequence of the moving plane procedure we also deduce that $u$ is monotone non-decreasing in $\Sigma_b$. Actually, arguing as in Lemma \ref{starting} we have $\partial_y u >0$ in $\Sigma_b$. In particular \emph {$u$ is positive} in the entire strip $\Sigma_{2b}.$ \\
Performing the moving plane method in the opposite direction $(0,-1)$ and \emph {observing that the corresponding $ \lambda^{*}$ is still equal to $2b$} (by the positivity of $u$),  we derive in the same way that
\[
u\geqslant u_b\qquad\text{in}\quad\Sigma_{b}\,,
\]
and this implies that $u$ is symmetric with respect to $\{y=b\}$.\\

Let us now consider the case $\lambda^*<2b$. In this case arguing as above we deduce that $u$ is positive in $\Sigma_{\lambda^*}$,  $u\leqslant u_\lambda$ in $\Sigma_\lambda$ for any $0<\lambda<\lambda^*/2$ and
\[
u\equiv u_{\lambda^*/2}\qquad\text{in}\quad\Sigma_{\lambda^*/2}\,.
\]
As above
$
\{y=\lambda^*\}\subseteq \{u=0\}\cap\{\nabla u=0\}\,,
$
that gives, by symmetry
\[
\{y=0\}\subseteq \{u=0\}\cap\{\nabla u=0\}\,.
\]
Now we deduce that $u$ is one-dimensional by the \emph{unique continuation principle} and consequently $u(x,y)\equiv u_0(y)$ for $u_0$ defined as in the statement.
\end{proof}

\begin{proof}[\underline{Proof of Theorem \ref{T:3+}}]
The proof is completely analogous to the proof of Theorem \ref{T:3}. Actually, it is easier in this case since we are assuming that $f(0)\geqslant 0$ so that the \emph{strong maximum principle} and the \emph{Hopf's Lemma} can be exploited. In particular we can start the moving plane procedure recovering \eqref{monteta2} via the \emph{Hopf's Lemma}.
Then  we complete the proof repeating verbatim the proof of Theorem \ref{T:3} and exploiting the fact that in this case, by the \emph{strong maximum principle}, the solution is either trivial or positive.
\end{proof}

\begin{proof}[\underline{Proof of Theorem \ref{T:3spe}}]
The proof of Theorem \ref{T:3spe} follows combining  Theorem \ref{T:3} and Theorem \ref{T:3+}.
\end{proof}

\section{Some results in any dimension $N \ge 2$}\label{higer}
In this section we state and prove some results for any dimension $ N \ge2$. We continue to assume that $f$ is locally Lipschitz continuous and satisfies $ f(0) <0$.

Let us consider the problem
\begin{equation}\label{E:PN}
\begin{cases}
-\Delta u=f(u) & \text{ in } \mathbb{R}^N_+\\
\quad u \geqslant 0 & \text{ in } \mathbb{R}^N_+\\
\quad u=0\,\, &\text{ on } \partial\mathbb{R}^N_+
\end{cases}
\end{equation}
and let us denote by $(x',y)$ a point in $\mathbb{R}^N$, with $x'=(x_1,\ldots,x_{N-1})$ and $y=x_N$.
For a fixed solution $u$, we consider the property $(\mathcal P_\mu)$ as in the introduction, and we have the following
\begin{thm}\label{dfbjkrgbkcvdvcgdcg}
Let $u\in C^2(\overline{\mathbb{R}^N_+})$ be a nonnegative solution to \eqref{E:PN} and let us set
\begin{equation}\label{LAMBDAbis}
\Lambda^*=\Lambda^*(u)\,:=\,\{\lambda>0\,\,:\,\,(\mathcal P_\mu)\quad\text{holds for every} \,\, 0<\mu\leq\lambda\}\,.
\end{equation}
Then
\[
\Lambda^*\,\ne\,\emptyset\,.
\]
\end{thm}
\begin{proof}
We prove the theorem by contradiction and therefore we assume that there exists a sequence of positive numbers $\mu_n$ such that,
$\mu_n$ tends to zero as $n\rightarrow \infty$ and $(\mathcal P_{\mu_n})$ fails, namely
\begin{equation}\nonumber
\{y=\mu_n\}\subset\{u=0\}\,.
\end{equation}

For $x'_0\in\R^{N-1}$ fixed, by the fact that $u$ is nonnegative,
 it follows that $\nabla u(x'_0,\mu_n)=0$. Exploiting the Dirichlet condition, it also follows that
 the real valued function $u(x'_0,t)$ with $t\in[0,\mu_n]$, has an interior local maximum  $t_n\in(0,\mu_n)$. Therefore
$ \partial_y u (x'_0,t_n)=0$.

\noindent By the mean value theorem we deduce that
\[
\partial_{yy} u (x'_0,\xi_n)=0\qquad\text{for some}\quad\xi_n\in[t_n\,,\,\mu_n]\,.
\]
Since $u\in C^2(\overline{\mathbb{R}^N_+})$, letting $n\rightarrow\infty$, we infer that
\[
\partial_{yy} u (x'_0,0)=0\,.
\]
This is a contradiction since in this case, recalling that $u=0$ in $\{y=0\}$, it  follows that
\[
-\Delta u (x'_0,0)=0> f(u(x'_0,0))
\]
and the result is proved.
\end{proof}

As a consequence of Theorem \ref{dfbjkrgbkcvdvcgdcg}, we see that
\begin{equation}\label{lambdabis}
\lambda^*=\lambda^*(u)\,:=\, \sup \Lambda^* \in (0, +\infty].
\end{equation}
We have the following
\begin{thm}\label{anydimension}
Let $u\in C^2(\overline{\mathbb{R}^N_+})$ be a nonnegative solution to \eqref{E:PN} and let $\lambda^*=\lambda^*(u)$
defined by \eqref{lambdabis}. Then, if $\lambda^*$ is finite, it follows that
  $u$ is one-dimensional and periodic, i.e.
\[
u(x,y)\,=\,u_0(y) \qquad \forall \, (x,y) \in \R^2_+
\]

where $ u_0\in C^2(\R, [0,\infty))$ is periodic of period $\lambda^* $ and is the unique solution of


\begin{equation}\label{oneDsup}
u_0''+f(u_0)=0\quad\text{in}\,\,[0,\infty)\qquad
u_0'(0)=u_0(0)=0=u_0'(\lambda^*)=u_0(\lambda^*)\,.
\end{equation}

Also, $u$ is symmetric with respect to $\{y=\frac{\lambda^*}{2}\}$ with $\partial_y u > 0$ in $\Sigma_{\frac{\lambda^*}{2}}$.
\end{thm}

\begin{proof}
It follows by the definition of $\lambda^*$ that
$
\{y=\lambda^*\}\subseteq \{u=0\}\,.
$
Since $u$ is nonnegative this implies that
\[
\{y=\lambda^*\}\subseteq \{\nabla u=0\}\,.
\]

To conclude, we argue as at the end of the proof of Theorem \ref{T:1}.
Indeed, for every $t \in \R^{N-1}$, the function $u^t(x',y) : = u(x'+t,y) $ is a nonnegative solution of \eqref{E:P} with $ u^t = \nabla u^t = 0$ on the set $ \{ y =\lambda^* \}$ and the unique continuation principle implies that $ u \equiv u^t $ on $\R^N_+$. This immediately gives that $u$ depends only on the variable $y$, i.e.,

\[
u(x',y)\,=\,u_0(y) \qquad \forall \, (x',y) \in \R^N_+
\]

where $u_0 \in C^2([0, +\infty)) $ is the unique solution of $u_0''+f(u_0)=0$ with $u_0'(0)=u_0(0)=0=u_0'(\lambda^*)=u_0(\lambda^*)$.

A simple analysis of \eqref{oneDsup} yields that $u_0$ is periodic of period $\lambda^*$ with $u_0$ even with respect to $\{y=\frac{\lambda^*}{2}\}$ and satisfying $u_0' >0$ in $(0, \frac{\lambda^*}{2})$. This concludes the proof.

\end{proof}



Now we turn to the case of coercive epigraphs and we prove Theorem \ref{T:Epi}.

\begin{proof}[\underline{Proof of Theorem \ref{T:Epi}}]

We plan to use the classical moving plane procedure \cite{S}.
We use the notation $\Sigma_\lambda:= \{(x',y) \in \R^N \, |\, 0<y<\lambda\}$ and we denote by $R_\lambda(x',y)$ the point symmetric to $(x',y)$ with respect to the hyperplane $\{y=\lambda\}$, namely
\[
R_\lambda(x',y)\,:=\, (x',2\lambda-y)\,.
\]
We set $u_\lambda(x',y)\,=\,u(R_\lambda(x',y))\,=\,u(x',2\lambda-y)$ and
$\Omega_\lambda\,:=\, \Omega\cap\Sigma_\lambda\,.$ Since $\Omega$ is a smooth coercive epigraph, we have that $\Omega_\lambda\,$ is a bounded open set (possibly non connected) satisfying $R_\lambda(\Omega_\lambda)\subset\Omega$,  for every $\lambda>0$.

Given any $\delta>0$, we can find $\lambda_0=\lambda_0(\delta)>0$ such that $\mathcal L (\Omega_\lambda)< \delta$ for any $0<\lambda\leqslant \lambda_0$. Therefore we can take $\delta$ small such that the \emph{weak comparison principle in small domains} (Proposition \ref{maxpri}) applies.
Since $u\leqslant u_\lambda$ on $\partial \Omega_\lambda$, Proposition \ref{maxpri} yields
\[
u\leqslant u_\lambda\qquad\text{in}\quad \Omega_\lambda\quad \text{for every}\,\,\,0<\lambda\leqslant \lambda_0\,.
\]
Therefore the set
\[
\Lambda\,:=\,\{\lambda>0\,\,:\,\,u\leqslant u_\mu\,\,\,\text{in}\,\,\Omega_\mu\quad\text{for any}\,\,0<\mu\leqslant\lambda\}
\]
is not empty and
\[
\bar\lambda\,:=\,\sup\,\,\Lambda \in (0, +\infty].
\]
Assume by contradiction that $ \bar \lambda <+\infty$. By continuity it follows that $u\leqslant u_{\bar\lambda}$ in $\Omega_{\bar\lambda}$.

Let us now prove that
\begin{equation}\label{disugstretta}
u<u_{\bar\lambda}\quad\text{in}\,\,\Omega_{\bar\lambda}.
\end{equation}

To this end, we observe that $u$ is \emph {positive} in a neighborhood of the boundary. Namely, for any $x\in\partial\Omega$, there exists $\rho=\rho(x)>0$ such that
\begin{equation}\label{positivadentro}
u>0\qquad\text{in}\quad \Omega\cap B_\rho(x)\,.
\end{equation}
A proof of this fact can be found, for instance, in \cite{BCN1} (cf. Lemma 4.1 on p. 485).

If \eqref {disugstretta} were false, there would exist a point $x_0 \in \Omega_{\bar\lambda}$ such that $ u(x_0) = u_{\bar\lambda}(x_0)$. Thence, the strong maximum principle would imply
\[
u \equiv u_{\bar\lambda}\quad\text{in} \quad \omega_{\bar\lambda},
\]

where $ \omega_{\bar\lambda}$ is the connected component of $\Omega_{\bar\lambda}$ containing the point $ x_0$. This clearly contradicts \eqref{positivadentro} (it would imply $u=0$ on $R_{\bar\lambda}(\partial\Omega\cap\omega_{\bar\lambda})$).  Hence \eqref {disugstretta} is satisfied.


In order to exploit Proposition \ref{maxpri} once again, let us fix
a bounded domain $D \subset \Omega$ containing the bounded set
$\Omega_{\bar \lambda +1} \cup R_{\bar\lambda + 1}(\Omega_{\bar\lambda +1})$ and then consider a compact set $\mathcal K\subset \Omega_{\bar\lambda}$ such that
\[
\mathcal L(\Omega_{\bar\lambda}\setminus \mathcal K)\leqslant \frac{\vartheta}{10}\,
\]

where $\vartheta$ is given by Proposition \ref{maxpri}.


It follows by compactness that, for some $ \sigma >0$,
\[
(u_{\bar\lambda}-u)\geqslant \sigma>0\qquad\text{on} \quad \mathcal K
\]
and therefore, by the uniform continuity of $u$ on compact sets, we can find $\varepsilon_0 \in (0,1)$ such that
\[
(u_{\bar\lambda+\varepsilon}-u)\geqslant \frac{\sigma}{2}>0\qquad\text{and}\qquad \mathcal L (\Omega_{\bar\lambda+\varepsilon}\setminus \Omega_{\bar\lambda})\leqslant \frac{\vartheta}{10}\quad\text{for every}\,\,\,\,0<\varepsilon<\varepsilon_0\,.
\]
Then it follows that Proposition \ref{maxpri} applies in $\Omega_{\bar\lambda+\varepsilon}\setminus\mathcal K$, for every
$0<\varepsilon <\varepsilon_0$, since
\[
\mathcal L(\Omega_{\bar\lambda+\varepsilon}\setminus\mathcal K) < \vartheta\,.
\]
Therefore $u\leqslant u_{\bar\lambda+\varepsilon}$ in $\Omega_{\bar\lambda+\varepsilon}\setminus\mathcal K$, for every
$0<\varepsilon <\varepsilon_0$, and consequently
\[
u\leqslant u_{\bar\lambda+\varepsilon}\qquad\text{in} \quad \Omega_{\bar\lambda+\varepsilon}\quad\text{for every}
\quad\,\,\,\,0<\varepsilon<\varepsilon_0\,.
\]
The latter contradicts the definition of $\bar\lambda$. Thus, we have proved that
\[
\bar \lambda=+\infty\,.
\]
As a consequence, $u$ is monotone non-decreasing in the $y$-direction. Actually, arguing exactly as above and using again that $u$ is \emph {positive} in a neighborhood of the boundary, we get
\begin{equation}
u<u_{\lambda} \qquad \text{in} \quad \Omega_{\lambda} \quad \text{for every}
\quad \lambda >0\,.
\end{equation}

Hence the Hopf's Lemma yields
\[
\partial_{x_N} u =\partial_y u>0\qquad\text{in}\quad\Omega\,.
\]
Furthermore, as a consequence, $u$ is positive in $\Omega$.
\end{proof}

\end{document}